\documentclass[a4paper,12pt,reqno]{amsart}
\usepackage{graphicx}

\usepackage[notref,notcite]{showkeys}
\usepackage{color}

\usepackage{amsmath,amstext,amssymb,amsopn,amsthm}
\usepackage{url}

\usepackage[margin=30mm]{geometry}
\usepackage{eucal,mathrsfs,dsfont}
\usepackage{soul}

\newtheorem{theorem}{Theorem}[section]
\newtheorem{corollary}[theorem]{Corollary}
\newtheorem{lemma}[theorem]{Lemma}

\theoremstyle{definition}

\theoremstyle{remark}

\newcommand{\calK}{\mathcal{K}}

\newcommand{\R}{\mathds{R}}

\newcommand{\N}{{\mathds{N}}}

\newcommand{\E}{\mathbb{E}}

\title  [Exit time]{ On the mean exit time from a ball for a symmetric  stable process}

\begin{document}
\author[M. Ryznar]{Micha{\l} Ryznar}

\thanks{ M. Ryznar was supported in part by the National Science Centre, Poland, grant no. 2023/49/B/ST1/03964, }
\address{Faculty of Pure and Applied Mathematics, Wroc{\l}aw University of Science and Technology, Wyb. Wyspia{\'n}skiego 27, 50-370 Wroc{\l}aw, Poland.}

\email{Michal.Ryznar@pwr.edu.pl}

\

\maketitle

\begin{abstract} Getoor in \cite{Getoor} calculated   the mean   exit time from a ball for  the standard isotropic $\alpha$-stable process in $\R^d$ starting from the interior of the ball. The purpose of this note is to show that, up to multplicative constant, the same formula is valid for any symmetric $\alpha$-stable process. 
\end{abstract}
\section{Introduction} Let $X_t$ be a $d$-dimensional  symmetric   $\alpha$-stable processes of index $\alpha \in (0,2)$, $d \in \N$. Let $\E^x = \E^x(\ \cdot\ |X_0=x)$ and let $B_r= \{z\in \R^d; |z|<r\}$. In this note we deal with the mean exit of the process $X$ from  $B_r$, that  is we study 
$\E^x \tau^X_{B_r},$
where 
$$\tau^X_{B_r}= \inf\{t>0, |X_{t}|\ge r\},\quad r>0.$$
Getoor in \cite{Getoor} studied  the mean   exit time from a ball in the case of the standard isotropic stable process $\tilde{X}$ and he calculated that 

$$\E^x \tau^{\tilde{X}}_{B_r}= C(\alpha,d) (r^2-|x|^2)_+^{\alpha/2}.$$

It is a bit surprising that the above formula holds (up to a multiplicative constant) for any symmetric stable process. Our main result is the following theorem. 

\begin{theorem} \label{main}Let  $\nu$ be the L\'evy measure  of a symmetric $\alpha$-stable process  $X_t$. Then 
$$ E^x\tau^X_{B_r} =\frac{\kappa_\alpha} {\nu(B_1^c)}  (r^2-|x|^2)_+^{\alpha/2},\quad x\in R^d,
$$
where $\kappa_\alpha=\frac1{\Gamma(1-\alpha/2)\Gamma(1+\alpha/2)}  $.
\end{theorem}

Our approach relies on the one-dimensional version of the result of \cite{Getoor}. Actually we apply the fact that the function $\R\ni u\to (r^2-|u|^2)_+^{\alpha/2}$ is in the domain of the generator  of the $L^1$ semigroup generated by the one-dimensional standard $\alpha$-stable process and the value of  the generator is a negative constant in the interval $(-r,r)$ (see \cite[Theorem 5.2.]{Getoor}). Next, using this result, we easily show that the value of a  generator of any $d$-dimensional symmetric stable process evaluated (pointwise) for a function $w(x) = (r^2-|x|^2)_+^{\alpha/2},\quad x\in R^d$ is equal to a negative constant if $|x|<r$. This allows to apply the local Ito formula (see eg. \cite[Lemma 3.8]{BKP}) to conclude our main result.

\section{Mean exit time calculations}
Let ${\mathcal S}= \{z\in \R^d; |z|=1\}$ be  the standard unit  sphere in $\R^d$
Let $\mu$ be a finite positive measure on $\mathcal S$.
Let $A(x), x\in \R^d$ be a family of $d\times d$   matrices. 
For a given $0<\alpha <2$ we consider the operator of the following form.
$$
\calK f(x) =      \int_{\mathcal S}\int_{\R}\left[f(x + A(x)z w) - f(x)\right] \,  \frac{dw}{|w|^{1 + \alpha}}\mu(dz),
$$
for any Borel function $f: \R^d \to \R$ and any $x \in \R ^d$ such that  the right hand side exists in the sense of principal value integral.

Next, for a given vector $v\neq 0$ in $\R^d$ let us introduce 
$$
\calK_v f(x) =     \int_{\R} \left[f(x + v w) - f(x)\right] \, \frac{dw}{|w|^{1 + \alpha}},  x\in\R^d, 
$$
again as the principal value integral. 
Let $v^*=v/|v|$. By a simple change of variables we have
\begin{equation}\label{scaling}
\calK_v f(x) =     |v|^\alpha\int_{\R} \left[f(x + v^* w) - f(x)\right] \, \frac{dw}{|w|^{1 + \alpha}}= |v|^\alpha\calK_{v^*} f(x).
\end{equation}
For $r>0$ we define
$$S_r(x)= c_\alpha (r^2-|x|^2)_+^{\alpha/2},\quad x\in \R^d, $$
where $$c_\alpha=\frac\alpha{2\Gamma(1-\alpha/2)\Gamma(1+\alpha/2)}.$$ 
The following simple  lemma is a key observation for the rest of this note.
\begin{lemma}\label{basic} If $v\in \R^d$, then 
$$\calK_{v} S_r(x)=-|v|^{\alpha}, |x|<r.$$
\end{lemma}\begin{proof}
By \eqref{scaling}, it is clear that it is enough to prove the lemma for $|v|=1$.
Next, we claim that  we can assume that $v$ can be taken as $e_1=(1,0,\dots,0)$. Indeed, if $R$ is a rotation such that $Rv=e_1$, then we have 
$$\calK_{v} S_r(x)= \calK_{v} (S_r\circ R) (x)= \calK_{e_1} S_r( R x), |x|<r.$$

Now, we recall
the one-dimensional version of the  result obtained by Getoor \cite[Theorem 5.2.]{Getoor}. Let $s_r(u)= c_\alpha(r^2-|u|^2)^{\alpha/2}_+, u\in \R$, then
for $-r<u<r$, where $r>0$
		\begin{equation}\label{Getoor}
\int_{\R} \left[s_r(u +  w) - s_r(u)\right] \, \frac{dw}{|w|^{1 + \alpha}}= -1.
\end{equation}
		Actually in \cite{Getoor} the above equality  was proved almost surely, but it is not difficult to show that the left hand side is a continuous function of  $u\in (-r,r)$, hence it must hold point-wise on $(-r,r)$.
	Next for $|x|<r	$ we have $|x_1|< \sqrt{r^2-|\tilde{ x}|^2}$, where $\tilde{ x}=x-x_1e_1$ and, by \eqref{Getoor}, we obtain
		\begin{eqnarray}
\calK_{e_1} S_r(  x)&=&c_\alpha\int_{\R} \left[(r^2-|x +  we_1|^2)^{\alpha/2}_+ - (r^2-|x|^2)^{\alpha/2}_+\right] \, \frac{dw}{|w|^{1 + \alpha}}\nonumber\\
&=&c_\alpha\int_{\R} \left[(r^2-|\tilde{ x}|^2-|x_1 +  w|^2)^{\alpha/2}_+ - (r^2-|\tilde{ x}|^2-|x_1|^2)^{\alpha/2}_+\right] \, \frac{dw}{|w|^{1 + \alpha}}\nonumber\\
&=&\int_{\R} \left[s_{\sqrt{r^2-|\tilde{ x}|^2}}(x_1 +  w) - s_{\sqrt{r^2-|\tilde{ x}|^2}}(x_1)
\right] \, \frac{dw}{|w|^{1 + \alpha}}= -1.\nonumber\end{eqnarray}
This implies that 
 $$\calK_{v} S_r(x)=\calK_{e_1} S_r( R x)=-1,$$
since $|R (x)|=|x|<r$. \end{proof}
\begin{corollary} \label{gen1}
$$
\calK S_r(x) =     - \int_{\mathcal S} |A(x)z|^{\alpha} \mu(dz), |x|<r.
$$
If $A(x)$  is an isometry for all $|x|<r$, then 
$$
\calK S_r(x) =    -  \mu(\mathcal S)=-|\mu|, |x|<r.
$$ 
\end{corollary}
\begin{proof}
Note that, due to Lemma \ref{basic}, we have
$$\calK S_r(x)= \int_{\mathcal S} \calK_{A(x)z}S_r(x)\mu(dz)=  -\int_{\mathcal S}|A(x)z|^{\alpha} \mu(dz).$$
\end{proof}

Now, we are in a position to calculate the mean exit time of a ball for arbitrary symmetric stable process. 
\begin{proof} [proof of Theorem \ref{main}] By the scaling property it is enough to prove the result for $r=1$. Let $\nu$ be the L\'evy measure of the process $X_t$, then we can represent $\nu$ in polar coordinates $(z,r), z\in \mathcal S, r>0$ as
$$\nu(dx)= \mu(dz)\frac {dr}{r^{1+\alpha}},$$ where the measure $\mu$ is called the spectral measure of the process.   
 By  $\calK_\nu$ we denote  the generator of the process.
\begin{eqnarray*}
\calK_\nu f(x) &=& p.v. \int_{\R^d}\left[f(x + y) - f(x)\right] \nu(dy)\\ 
&=&\int_{\R^d}\left[f(x +y) - f(x)- \nabla f(x)\cdot y {\mathbf 1}_ {|y|\le 1} \right]\nu(dy)
 \end{eqnarray*}
We observe that the above  integral is finite for any $x$  such that $f$ is locally $C^2$ class in a neighborhood of $x$ and bounded on $\R^d$.
Using  polar coordinates and symmetry we can represent $\calK_\nu$ as
\begin{eqnarray*}\calK_\nu f(x) &=& \frac12\int_{S}\int_{\R}\left[f(x +z w) - f(x )\right] \,  \frac{dw}{|w|^{1 + \alpha}}\mu(dz)\\,
&=& \frac12\int_{S}\int_{\R}\left[f(x +z w) - f(x)- \nabla f(x)\cdot z w {\mathbf 1}_ {|w|\le 1} \right] \,  \frac{dw}{|w|^{1 + \alpha}}\mu(dz)\end{eqnarray*}
Let $0<r<1$. By $\tau_{r}$ we denote $\tau^X_{B_r}$.  Then $S=S_1$ is of $C^2(B_{(1+r)/2})$ class (all first and second order partial derivatives are bounded in $B_{(1+r)/2}$).
Since  $X_t$ is a pure jump process, from  the local Ito formula  formula ( see \cite[Lemma 3.8.]{BKP}) it follows that  
\begin{eqnarray}\label{Ito}S(X_{t \wedge \tau_{r}} )-S(X_{0})&=& 
\int_{0^+}^{t\wedge \tau_{r}}\nabla S(X_{u-})\cdot dX_u\nonumber\\&+&\sum_{0<u\le t\wedge \tau_{r}, |\Delta X_u|>0 }S(X_{u} )-S(X_{{u-}})-\nabla S(X_{u-})\cdot\Delta X_u\nonumber\\&:=& I_t+J_t.\end{eqnarray}
  We can split the intgral $I$ 
as the sum of two itegrals where we consider small and large jumps. Let $X^1_t$ and $X^2_t$ be two independent processes with L\'evy measures $\nu_1 = \nu|_{\overline{B}_1}$ and $\nu_2=\nu-\nu_1$, respectively. Then we can take 
$$X_t= X^1_t+X^2_t$$
and then we have 
$$I_t= \int_{0^+}^{t\wedge \tau_{r}}\nabla S(X_{u-})dX^1_u+ \int_{0^+}^{t\wedge \tau_{r}}\nabla S(X_{u-})dX^2_u:=I_t^1+I_t^2$$
Note that $X^2_t$ is a compound Poisson process and its jumps are the big jumps of the process $X_t$, hence 
 $$I_t^2 =\sum_{0<u\le t\wedge \tau_{r}, |\Delta X_u|> 1 }\nabla S(X_{u-})\cdot\Delta X_u.$$
Therefore we can rewrite \eqref{Ito} as 
\begin{eqnarray}\label{Ito1}S(X_{t \wedge \tau_{r}} )-S(X_{0})&=& 
I_t^1+\sum_{0<u\le t\wedge \tau_{r}, \Delta X_u\ne 0 }S(X_{u} )-S(X_{{u-}})-\nabla S(X_{u-})\cdot\Delta X_u\mathbf{1}_{|\Delta X_u|\le 1}\nonumber\\&:=& I_t^1+J_t^1.\end{eqnarray}

We claim that $I_t^1$ being a local martingale  is in fact a martingale.   This follows (see \cite[Corrolary 3, p.73 ] {P} ) since its quadratic variation process  $[I^1, I^1]_t$ is integrable.  Indeed, we have

\begin{eqnarray*}\E^x[I^1, I^1]_t&\le& \E^x d\sum_{u\le t\wedge \tau_{r}, |\Delta X_u|\le 1 } \max_{v\le t\wedge \tau_{r}}(|\nabla S(X_{v-})|^2)|\Delta X_u|^2\\
&\le& d C(r) \left(|x|^2+\E^x\sum_{0<u\le t, |\Delta X_u|\le 1 } |\Delta X_u|^2\right)\\
&=& dC(r)\left(|x|^2 +t\int_{|y|\le 1} |y|^2\nu(dy)\right)< \infty, \end{eqnarray*}
 where
$$C(r)= \max_{|x|<r} |\nabla S(x)|^2<\infty.$$

Next, using  L\'evy's system formula (see   \cite[Lemma 4.7]{BRSW} with $g(u)= \mathbf1_{(0,t\wedge \tau_{r}]}(u)$ ) we have 
\begin{equation}\label{Levy system}
\E^x J^1_t= \E^x\int_0^{t\wedge \tau_{r}} du\int_{\R^d}\left( S(X_{u}+y)-S(X_{{u}})-\nabla S(X_{u})\cdot y 
 \mathbf{1}_{|y|\le1}\right)\nu(dy).
\end{equation}
To justify that  we can use   L\'evy's system formula  we observe that 
$$\left|S(x+y)-S(x)-\nabla S(x)\cdot y \mathbf{1}_{|y|\le 1}\right|\le c \min(|y|^2,1),$$ if $|x|\le r $,
where $c=c(r)$.

Note that for $u< \tau_{r}$
$$\int_{\R^d}\left( S(X_{u}+y)-S(X_{{u}})-\nabla S(X_{u})\cdot y 
 \mathbf{1}_{|y|\le1}\right)\nu(dy)=\calK_\nu S(X_{{u}})=-\frac12|\mu|,$$
where the last equality follows from Corollary \ref{gen1} if we take $A(x)$ to be the identity matrix for all $x\in \R^d$.
Hence, by \eqref{Levy system}, we obtain
$$\E^x J^1_t= - \frac12|\mu| \E^x \tau_{r}$$
Since $\E^x I^1_t= 0$, as $I^1_t$ is a   martingale, by \eqref{Ito1}, we arrive at
$$\E^x(S(X_{t\wedge \tau_{r}} ))-S(x)= - \frac12|\mu|\E^x ( t\wedge \tau_{r}).$$
Letting $t\to \infty $, by bounded convergence theorem, 
$$\E^x(S(X_{ \tau_{r}} ))= S(x) - \frac12|\mu| \E^x \tau_{r}.$$
Observing that  $$\E^x(S(X_{ \tau_{r}} ))\le c_\alpha(1- r^2)^{\alpha/2}$$
and letting $r\uparrow1$, we get
$$S(x)= \frac12|\mu|\lim_{r\uparrow1}\E^x \tau_{r}= \frac12 |\mu|\E^x \tau_{1}.$$
Next we note that $$\nu(|y|\ge 1)= |\mu|\int_1^\infty \frac {dr}{r^{1+\alpha}}= \frac{|\mu|}{\alpha}.$$
Since $\frac{2c_\alpha}\alpha= \kappa_\alpha$ 
the proof is completed. 
\end{proof}

\textbf{ Acknowledgements.} 
The author is very grateful to K. Bogdan and T. Kulczycki  for discussions on the problem treated in the paper.

\end{document}